\documentclass[12pt]{article}
\usepackage{}

\usepackage{epsfig}
\usepackage{latexsym}
\usepackage{caption}
\usepackage{amsfonts}
\usepackage{amssymb}
\usepackage{mathrsfs}
\usepackage{amsmath}
\usepackage{enumerate}
\usepackage{graphics}
\usepackage{MnSymbol}
\usepackage{float}
\usepackage{pict2e}
\usepackage{subfigure}

\usepackage{color}

\makeatletter

\newcommand{\Rmnum}[1]{\expandafter\@slowromancap\romannumeral #1@}

\makeatother

\begin{document}

\newtheorem{theorem}{Theorem}
\newtheorem{observation}[theorem]{Observation}
\newtheorem{corollary}[theorem]{Corollary}
\newtheorem{algorithm}[theorem]{Algorithm}
\newtheorem{definition}{Definition}
\newtheorem{guess}{Conjecture}
\newtheorem{claim}[theorem]{Claim}
\newtheorem{problem}[theorem]{Problem}
\newtheorem{question}[theorem]{Question}
\newtheorem{lemma}[theorem]{Lemma}
\newtheorem{proposition}[theorem]{Proposition}
\newtheorem{fact}[theorem]{Fact}

\makeatletter
  \newcommand\figcaption{\def\@captype{figure}\caption}
  \newcommand\tabcaption{\def\@captype{table}\caption}
\makeatother

\newtheorem{acknowledgement}[theorem]{Acknowledgement}

\newtheorem{axiom}[theorem]{Axiom}
\newtheorem{case}[theorem]{Case}
\newtheorem{conclusion}[theorem]{Conclusion}
\newtheorem{condition}[theorem]{Condition}
\newtheorem{conjecture}{Conjecture}
\newtheorem{criterion}[theorem]{Criterion}
\newtheorem{example}[theorem]{Example}
\newtheorem{exercise}[theorem]{Exercise}
\newtheorem{notation}{Notation}
\newtheorem{solution}[theorem]{Solution}
\newtheorem{summary}[theorem]{Summary}

\newenvironment{proof}{\noindent {\bf
Proof.}}{\rule{3mm}{3mm}\par\medskip}
\newcommand{\remark}{\medskip\par\noindent {\bf Remark.~~}}
\newcommand{\pp}{{\it p.}}
\newcommand{\de}{\em}
\newcommand{\mad}{\rm mad}
\newcommand{\qf}{Q({\cal F},s)}
\newcommand{\qff}{Q({\cal F}',s)}
\newcommand{\qfff}{Q({\cal F}'',s)}
\newcommand{\f}{{\cal F}}
\newcommand{\ff}{{\cal F}'}
\newcommand{\fff}{{\cal F}''}
\newcommand{\fs}{{\cal F},s}
\newcommand{\s}{\mathcal{S}}
\newcommand{\G}{\Gamma}
\newcommand{\g}{(G_3, L_{f_3})}
\newcommand{\wrt}{with respect to }
\newcommand {\nk}{ Nim$_{\rm{k}} $  }
\newcommand {\dom}{ {\rm Dom}  }
 \newcommand {\ran}{ {\rm Ran}  }

\newcommand{\ch}{{\rm ch}}
\newcommand{\ee}{{\epsilon}}

\newcommand{\q}{\uppercase\expandafter{\romannumeral1}}
\newcommand{\qq}{\uppercase\expandafter{\romannumeral2}}
\newcommand{\qqq}{\uppercase\expandafter{\romannumeral3}}
\newcommand{\qqqq}{\uppercase\expandafter{\romannumeral4}}
\newcommand{\qqqqq}{\uppercase\expandafter{\romannumeral5}}
\newcommand{\qqqqqq}{\uppercase\expandafter{\romannumeral6}}

\newcommand{\qed}{\hfill\rule{0.5em}{0.809em}}

\newcommand{\var}{\vartriangle}

\title{{\large \bf A refinement of choosability  of graphs }}
%\title{{The properties of the hypercube-like}}

\author{  Xuding Zhu\thanks{Department of Mathematics, Zhejiang Normal University,  China.  E-mail: xudingzhu@gmail.com. Grant Numbers: NSFC 11571319 and 111 project of Ministry of Education of China.}}

\maketitle

 {\tiny }

\begin{abstract}
Assume $k$ is a positive integer, $\lambda=\{k_1, k_2, \ldots, k_q\}$ is a partition of $k$ and $G$ is a graph. A $\lambda$-assignment of $G$ is a $k$-assignment $L$ of $G$ such that the colour set $\bigcup_{v\in V(G)}L(v)$ can be partitioned into $q$ subsets $C_1 \cup C_2 \ldots \cup C_q$ and for each vertex $v$ of $G$, $|L(v) \cap C_i| =  k_i$. We say $G$ is $\lambda$-choosable if for each $\lambda$-assignment $L$ of $G$, $G$ is $L$-colourable.
It follows from the definition that if $\lambda =\{k\}$, then $\lambda$-choosability  is the same as $k$-choosability, if $\lambda =\{1,1,\ldots, 1\}$, then $\lambda$-choosability  is equivalent to $k$-colourability. For the other partitions of $k$ sandwiched between 
$\{k\}$ and $\{1,1,\ldots, 1\}$ in terms of refinements, $\lambda$-choosability reveals a complex hierarchy of colourability of graphs.  
 We prove that for two partitions  $\lambda, \lambda'$ of $k$,  every $\lambda$-choosable graph is $\lambda'$-choosable if and only if $\lambda'$ is a refinement of $\lambda$. Then we study $\lambda$-choosability of special families of graphs. 
 The Four Colour Theorem 
 says that every planar graph is $\{1,1,1,1\}$-choosable.
 A very recent result of Kemnitz and Voigt implies that for any partition $\lambda$ of $4$ other than $\{1,1,1,1\}$, there is a planar graph which is not $\lambda$-choosable. 
  We observe that, in contrast to the fact that there are non-$4$-choosable  $3$-chromatic planar graphs,   every $3$-chromatic planar graph is $\{1,3\}$-choosable, and that if $G$ is a planar graph  whose dual $G^*$    has a connected spanning Eulerian subgraph,   then $G$ is $\{2,2\}$-choosable. We prove that if $n$ is a positive even integer,    $\lambda$ is a partition of $n-1$ in which each part is at most $3$, then $K_n$ is edge $\lambda$-choosable. 
  Finally we study relations between $\lambda$-choosability of graphs and colouring of signed graphs and generalized signed graphs. A conjecture of M\'{a}\v{c}ajov\'{a},   Raspaud and \v{S}koviera that every planar graph is signed 4-colcourable is recently disproved by Kardo\v{s}  and Narboni. We prove that every signed $4$-colourable graph is weakly $4$-choosable, and every signed $Z_4$-colourable graph is 
 $\{1,1,2\}$-choosable. The later result combined with the above result of Kemnitz and Voigt disproves a conjecture of Kang and Steffen that every planar graph is signed $Z_4$-colourable. We shall show   that 
 a graph constructed by Wegner in 1973 is also a counterexample to Kang and Steffen's conjecture, and present a new construction of a non-$\{1,3\}$-choosable planar graphs.


\end{abstract}

%% \linenumbers

\section{Introduction}

A proper colouring of a graph $G$ is a mapping $f$ which assigns to each vertex $v$ a colour such that colours assigned to   adjacent vertices are distinct. A {\em $k$-colouring} of $G$ is a proper colouring $f$ of $G$ such that $f(v) \in \{1,2,\ldots, k\}$ for each vertex $v$. The {\em chromatic number} $\chi(G)$ of $G$ is the minimum integer $k$ such that $G$ is $k$-colourable.

An {\em   assignment} of a graph $G$ is a mapping $L$ which assigns to each vertex $v$ of $G$ a set
$L(v)$ of   permissible colours. A  {\em proper $L$-colouring} of $G$ is a proper colouring $f$ of $G$ such that for each vertex $v$ of $G$, $f(v) \in L(v)$. We say $G$ is {\em $L$-colourable} if $G$ has a proper  $L$-colouring. A {\em $k$-assignment} of $G$ is a assignment $L$ with $|L(v)|= k$ for each vertex $v$.  We say $G$ is {\em $k$-choosable} if $G$ is $L$-colourable for any $k$-assignment $L$ of $G$. The {\em choice  number} $ch(G)$ of $G$ is the minimum integer $k$ such that
$G$ is $k$-choosable.

The concept of list colouring was introduced by Erd\H{o}s, Rubin and Taylor \cite{ERT}, and independently by Vizing \cite{Vizing} in the 1970's, and provides a useful tool in many inductive proofs for upper bounds for the chromatic number of graphs. 
On the other hand, there is a big gap between $k$-colourability and $k$-choosability. In particular, bipartite graphs can have arbitrary large choice number.

 Intuitively, the reason that   a   $k$-colourable graph fails to be $L$-colourable for a $k$-assignment $L$    is due to the fact that lists assigned to vertices by $L$ may be complicately entangled. In this paper, we  put restrictions on the entanglements of lists that are allowed to be assigned to the vertices,
 and hence builds a refined scale for measuring choosability of graphs.

\begin{definition}
	\label{def-partition}
	By a partition of a positive integer $k$ we mean a finite multiset $\lambda =  \{k_1,k_2,\ldots, k_q\}$   of positive integers with $k_1+k_2+ \ldots + k_q = k$. 
	Each integer $k_i \in \lambda$ is called a {\em part} of $\lambda$.
\end{definition}

\begin{definition}
	\label{def-multiset}
	Assume $\lambda =\{k_1, k_2, \ldots, k_q\}$ is a partition of $k$ and $G$ is a graph. A {\em $\lambda$-assignment}  of $G$ is a  $k$-assignment $L$ of $G$ in which the colours in $\bigcup_{x \in V(G)}L(x)$ can be partitioned into sets $C_1, C_2, \ldots, C_q$ so that for each vertex $x$ and for each $1 \le i \le q$, $|L(x) \cap C_i| = k_i$. Each $C_i$ is called a {\em colour group} of $L$. We say $G$ is {\em $\lambda$-choosable}  if $G$ is $L$-colourable for any $\lambda$-assignment $L$ of $G$. 
\end{definition}

Equivalently, for a partition $\lambda = \{k_1,k_2,\ldots, k_q\}$ of $k$, a $k$-assignment $L$ of $G$ is a $\lambda$-assignment   of $G$ if for each $i=1,2,\ldots,q$, there is $k_i$-assignment $L_i$ of $G$ such that $L = \bigcup_{i=1}^q L_i$ (i.e., for each vertex $x$ of $G$, $L(x) = \bigcup_{i=1}^qL_i(x)$) and for $i \ne j$, for any vertices $x,y$ of $G$, $L_i(x) \cap L_j(y) = \emptyset$.

Assume $\lambda$ is a partition  of $k$. By {\em subdividing a part} of $\lambda$, we mean replacing a part $k_i \in \lambda$ with a few parts 
that form a partition of $k_i$. 
Assume $\lambda$ and $\lambda'$ are two partitions of $k$. We say $\lambda'$ is a {\em refinement} of $\lambda$ if $\lambda'$ is obtained from   $\lambda$ by subdividing some parts of $\lambda$. 
For example, $\lambda'=\{2,3,4\}$ is a refinement of $\lambda=\{4,5\}$. It follows from the definition that if $\lambda'$ is a  refinement  of $\lambda$, then every $\lambda'$-assignment of a graph $G$ is also a $\lambda$-assignment of $G$. Hence every $\lambda$-choosable graph is 
$\lambda'$-choosable.

\begin{definition}
	Assume $\lambda=\{k_1, k_2, \ldots, k_q\}$ is a partition of $k$ and $L$ is a $\lambda$-assignment of $G$ and  $C = \bigcup_{v \in V(G)}L(v)=C_1 \cup C_2 \cup \ldots \cup C_q$ is a   partition of the colour set into colour groups of $L$.  
	If for each $k_i=1$, the corresponding colour group $C_i$ is a singleton, then we say $L$ is a special $\lambda$-assignment.
\end{definition}

\begin{lemma}
	\label{lem0}
	Assume $\lambda=\{k_1, k_2, \ldots, k_q\}$ is a partition of $k$. A graph $G$ is $\lambda$-choosable if and only if for any special $\lambda$-assignment $L$ of $G$, $G$ is $L$-colourable.
\end{lemma}
\begin{proof}
	If $G$ is $\lambda$-choosable, then of course for any special $\lambda$-assignment $L$, $G$ is $L$-colourable.
	
	Assume $G$ is $L$-colourable for any special $\lambda$-assignment $L$, and $L'$ is an arbitrary $\lambda$-assignment of $G$. Let $J = \{i: k_i = 1\}$. 
	Assume $C_1, C_2 , \ldots, C_q$ are the colour groups of $L$.
	For each $i \in J$, let $c_i$ be an arbitrary colour in $C_i$. Let $L(v) = (L'(v) - \bigcup_{i \in J}C_i) \cup \bigcup_{i\in  J}\{c_i\}$. Then $L$ is a special $\lambda$-list assgnment of $G$. By assumption, $G$ has a proper $L$-colouring $\phi$. For each vertex $v \in V(G)$ and each index $i \in J$, let $c_{v,i}$ be the unique colour in $L'(v) \cap C_i$. Let 
	\[
	\phi'(v) = \begin{cases} c_{v,i}, &\text{if $i \in J$ and $\phi(v)=c_i$},\cr
	\phi(v), &\text{otherwise}.
	\end{cases}
	\]
	Then $\phi'$ is a proper $L'$-colouring of $G$.
\end{proof}

It follows from the definition that for a  positive integer $k$ and a graph $G$, a $\{k\}$-assignment is the same as a $k$-assignment.
Hence $\{k\}$-choosable is the same as $k$-choosable. On the other hand, it follows from Lemma \ref{lem0}  that if $\lambda=\{1,1,\ldots,1\}$ is the multiset consisting of $k$ copies of $1$, then $\lambda$-choosable is the same as $k$-colourable. So $\lambda$-choosability puts $k$-colourability and $k$-choosability of graphs under the same framework, and $\lambda$-choosability for those partitions $\lambda$ of $k$   sandwiched between $\{k\}$ and $\{1,1,\ldots, 1\}$ (in terms of refinements) reveal a complicated hierarchy of colourability of graphs.

\begin{definition}
	\label{def-order}
	Assume $\lambda$ is a partition of $k$ and $\lambda'$ is a partition of $k' \ge k$. We write $\lambda \le \lambda'$ if $\lambda'$ is a refinement of a partition $\lambda''$ of $k'$ which is obtained from $\lambda$   by increasing some of   parts of $\lambda$.
\end{definition}
  
  For example, $\lambda=\{2,2\}$ is a partition of $4$, and $\lambda'=\{1,1,1,3\}$ is a partition of $6$. Let $\lambda''=\{3,3\}$.
  Then $\lambda''$ is obtained from $\lambda$ by increasing each part of $\lambda$ by $1$, and $\lambda'$ is a refinement of $\lambda''$. Hence $\lambda \le \lambda'$. 
  
  If $\lambda''$ is obtained from $\lambda$ by increasing some of parts of $\lambda$, then certainly every $\lambda$-choosable graph is $\lambda''$-choosable. If $\lambda'$ is a refinement of $\lambda''$, then every $\lambda''$-choosable graph is $\lambda'$-choosable. Therefore   if $\lambda \le \lambda'$, then every $\lambda$-choosable graph is $\lambda'$-choosable. 
 
In Section 2, we shall prove the converse of 
the above observation is also true:   If
 every $\lambda$-choosable graph is $\lambda'$-choosable, then $\lambda \le \lambda'$.
  
In Section 3, we study $\lambda$-choosability of planar graphs and line graphs. It is known that every planar graph is $5$-choosable \cite{Tho1994} and there are planar graphs that are not $4$-choosable \cite{Voigt}.  By the four colour theorem, every planar graph is $\{1,1,1,1\}$-choosable. A very recent result of Kemnitz and Voigt \cite{KV2018} shows that there are planar graphs that are not  $\{1,1,2\}$-choosable. This implies that for any partition $\lambda$ of $4$ different from $\{1,1,1,1\}$, there is a planar graph which is not $\lambda$-choosable. I.e., the Four Colour Theorem is tight in the refined scale of choosability defined in this paper. Nevertheless, many interesting problems concerning $\lambda$-choosability of subfamilies of planar graphs remains open.  
 Mirzakhani \cite{Mirzakhani} constructed a $3$-chromatic planar graph which is not $4$-choosable. In contrast to this result,  we observe that $3$-chromatic planar graphs are $\{1,3\}$-choosable.
We also show that if $G$ is a  planar graphs   whose dual $G^*$ contains a spanning Eulerian subgraph $H$ such that every  face of $H$ is either incident to a single connected component of $H$ or incident to   two connected components of $H$ that are joined by an even number of edges in $G^*$, then $G$ is $\{2,2\}$-choosable. In particular, if $G^*$ has a connected spanning Eulerian subgraph, then $G$ is $\{2,2\}$-choosable. 
It  remains an open problem as whether every $3$-chromatic planar graph is $\{2,2\}$-choosable. We also present in this section a new construction of a planar graph which is not $\{1,3\}$-choosable. Then we prove that if $n$ is an even integer, and $\lambda$ is a partition of $n-1$ in which each part is at most $3$, then $K_n$ is edge $\lambda$-choosable.
% It is also an open problem as whether every triangle free planar graph is $\{1,2\}$-choosable.
% As a consequence of this result, we show that every $3$-chromatic planar graph is $\{2,2\}$-choosable. In contrast to this result, it is known that there are $3$-chromatic planar graphs that are not $4$-choosable \cite{Mirzakhani}. 
 
In Section 4, we discuss relation between $\lambda$-choosability and colouring of signed graphs and generalized signed graphs. 
A signed graph is a pair $(G, \sigma)$ such that $G$ is a graph and $\sigma: E \to \{-1,+1\}$ is a {\em signature} which assigns to each edge a sign. 
Colouring of signed graphs was first studied by Zalslavsky \cite{Z} in the 1980's and has attracted a lot of recent attention \cite{MRS,KS2, KS}. 
A set $I$ of integers is called {\em symmetric} if for any integer $i$, $i \in I$ implies that $-i \in I$. 
For a positive integer $k$, let $Z_k$ be the cyclic group of order $k$, and let $N_k$ be a symmetric set  of $k$ integers, say $N_k=\{1, -1, 2, -2, \ldots, q, -q\}$ if $k=2q$ is even and $N_k = \{0, 1, -1, 2, -2, \ldots, q, -q\}$ if $k=2q+1$ is odd. A  {\em   $k$-colouring} of $(G, \sigma)$ is a mapping $f: V(G) \to N_k$ such that for each edge $e=xy$, $f(x) \ne \sigma(e)f(y)$, and 
a {\em   $Z_k$-colouring} of $(G, \sigma)$ is a mapping $f: V(G) \to Z_k$ such that for each edge $e=xy$, $f(x) \ne \sigma(e)f(y)$. 
We say a graph $G$ is {\em signed $k$-colourable} (respectively, {\em signed  $Z_k$-colourable}) if for any signature $\sigma$ of $G$, the signed graph $(G,\sigma)$ is $k$-colourable (respectively,  $Z_k$-colourable).  It was conjectured by M\'{a}\v{c}ajov\'{a},   Raspaud and \v{S}koviera \cite{MRS} that every planar graph is signed $4$-colourable, and conjectured by Kang and Steffen \cite{St} that every planar graph is signed $Z_4$-colourable. 
An assignment $L$ of a graph $G$ is called {\em symmetric} if for each vertex $v$ of $G$, $L(v)$ is a symmetric set of integers. We say $G$ is {\em weakly $k$-choosable} if $G$ is $L$-colourable for any symmetric $k$-assignment $L$. K\"{u}ndgen and Ramamurthi \cite{KR2002} conjectured that every planar graph is weakly $4$-choosable. It follows from the definition that every $\{2,2\}$-choosable  graph is 
  weakly $4$-choosable. We prove that every signed $4$-colourable graph is also weakly $4$-choosable, and that every signed $Z_4$-colourable planar graph is $\{1,1,2\}$-choosable. So M\'{a}\v{c}ajov\'{a},   Raspaud and \v{S}koviera's conjecture implies  K\"{u}ndgen and Ramamurthi's conjecture. Very recently, Kardo\v{s} and Narboni \cite{KN} disproved the conjecture of M\'{a}\v{c}ajov\'{a},   Raspaud and \v{S}koviera. The conjecture of 
  K\"{u}ndgen and Ramamurthi's remains open.
  On the other hand,  our result combined with the above mentioned result of Kemnitz and Voigt \cite{KV2018} disproves the conjecture of Kang and Steffen. We shall also present a direct construction of a signed planar graph which is not $Z_4$-colourable. This example is obtained by assigning signs to edges of  a   planar graph constructed by Wegner in 1973 \cite{Wegner}. 
 
Section 5 studies colouring of generalized signed graphs.
DP-colouring of graphs is a concept introduced by Dvo\v{r}\'{a}k and Postle \cite{DP} as a variation of list colouring of graphs. 
In the same spirit of generalizing choosabilities to $\lambda$-choosabilities, DP-colouring is generalized to colouring of generalized signed graphs. Relation between colouring of such generalized signed graphs and $\lambda$-choosabilities is discussed in this section.

In Section 6, we collect some open problems concerning $\lambda$-choosability of graphs. 
 
\section{Ordering of partitions of integers}
\label{sec-order}

We have defined a  relation $\le$ on the set of partitions of integers, which is easily seen to be   reflexive, anti-symmetric  and transitive. I.e., $\le$ is a partial ordering of the partitions of positive integers. 
This section proves the following result.

\begin{theorem}
	\label{main1}
	If $\lambda \le \lambda'$, then every $\lambda$-choosable graph is $\lambda'$-choosable, and conversely, if every $\lambda$-choosable graph is $\lambda'$-choosable, then $\lambda \le \lambda'$. 	
\end{theorem}

Before proving this theorem, we prove a lemma, which will be used in our proof, and which is of independent interest.

For graphs $G_1, G_2, \ldots, G_q$, the {\em join}  $ \vee_{i=1}^q G_i$  of    $G_1, G_2, \ldots, G_q$   is obtained from the disjoint union of the $G_i$'s   by adding edges connecting every vertex of $G_i$ to every vertex of $G_j$ for any $i \ne j$.

\begin{lemma}
	\label{lem1}
	Assume for $i=1,2,\ldots, q$, $\lambda_i$ is a partition of $k_i$, and
	$G_i$ is $\lambda_i$-choosable. Let $\lambda = \bigcup_{i=1}^q\lambda_i$ be the union of $\lambda_i$. Then $\vee_{i=1}^qG_i$ is $\lambda$-choosable.  In particular, if $G_i$ is $k_i$-choosable, and $\lambda=\{k_1, k_2, \ldots, k_q\}$, then $\vee_{i=1}^qG_i$ is $\lambda$-choosable.
\end{lemma}
\begin{proof}
	Let $L$ be a $\lambda$-assignment of $G$. Then $L(v)$ is the disjoint union of $L_1(v) \cup L_2(v) \cup \ldots \cup L_q(v)$, such that $L_i$ is a $\lambda_i$-assignment of $G$ and $L_i(v) \cap L_j(v') = \emptyset$ for any $i \ne j$. Let $f_i$ be an $L_i$-colouring of $G_i$, then the union of the $f_i$'s is an $L$-colouring of $G$. 
\end{proof}

\bigskip
\noindent
{\bf Proof of Theorem \ref{main1}}

	One direction of this theorem follows easily from the definitions.
	Assume $\lambda$ is a partition of $k$, $\lambda'$ is a partition of $k'$ and $\lambda \le \lambda'$ and $G$ is $\lambda$-choosable.
	It follows from the definition that there is a partition $\lambda''$ of $k'$ which is obtained from $\lambda$ by increasing some parts of $\lambda$, and $\lambda'$ is a refinement of $\lambda''$. 
	Let $L$ be  a $\lambda'$-assignment of $G$. Then $L$ is also a $\lambda''$-assignment of $G$. By omitting some colours from $L(v)$ for each vertex $v$ of $G$ (if needed), we obtain a $\lambda$-assignment $L'$ of $G$. Since  $G$ is $\lambda$-choosable,   $G$ has an $L'$-colouring, which is also an $L$-colouring of $G$. Hence $G$ is $\lambda'$-choosable.

 Assume that every $\lambda$-choosable graph is $\lambda'$-choosable. We shall prove that $\lambda \le \lambda'$. 
	
	Assume $\lambda=(k_1,k_2, \ldots, k_q)$ and $\lambda'=(k'_1,k'_2, \ldots, k'_p)$. 
		We construct a graph $G$ as follows: Let $n$ be a sufficiently large integer (to be determined later). For $1\le i \le q$, let $G_i$ be the disjoint union of $n$ copies of the complete graph  $K_{k_i}$ on $k_i$-vertices. Let $G= \vee_{i=1}^q G_i$. 
	
	Since each $G_i$ is $k_i$-choosable, it follows from Lemma \ref{lem1} that   $G$ is  $\lambda$-choosable. 
	
	By our assumption, $G$ is  $\lambda'$-choosable. 
	
	Assume $C'_1, C'_2, \ldots, C'_p$ are disjoint colour sets such that each $|C'_j| = 2k'_j-1$.
	
	Let
	$${\cal S} = {C'_1 \choose k'_1} \times {C'_2 \choose k'_2} \times \ldots \times {C'_p \choose k'_p}.$$
	
	Here ${C'_j \choose k'_j}$ is the family of all $k'_j$-subsets of $C'_j$.
	
	Let $n=|{\cal S}|$.  Assume ${\cal S} =\{S_1, S_2, \ldots, S_n\}$ and   $S_j = (S_{j,1},S_{j,2},\ldots, S_{j,p}) $, where $S_{j,i} \in {C'_i \choose k'_i}$.

	Note that $G_i$ contains $n$ copies of $K_{k_i}$, which are labeled as the 1st copy, the 2nd copy, etc. of $K_{k_i}$. 
	 
	For each vertex $v$ of the $j$th copy of $K_{k_i}$ in $G_i$,
	let $L(v) = \bigcup_{i=1}^qS_{j,i}$. Then $L$ is a $\lambda'$-assignment  of $G$.
	
	By our assumption, there is an   $L$-colouring $\phi$ of $G$. For each index $j \in \{1,2,\ldots, p\}$, we say $C'_j$ is {\em occupied   by   $G_i$}  if at least $k'_j$ colours in $C'_j$ are used by vertices in $G_i$. 
	
	For each $i \in \{1,2,\ldots, q\}$, let $$J_i = \{j: \text{ $C'_j$ is occupied by  $G_i$}\}.$$

	\begin{claim}
		\label{cl1}
		For $i, i' \in \{1,2,\ldots, q\}$ and $i \ne i'$, we have $J_i \cap J_{i'} = \emptyset$.
	\end{claim}
	\begin{proof}
		If $J_i \cap J_{i'} \ne \emptyset$, say $j \in J_i \cap J_{i'}$, then at least $k'_j$ colours are used by vertices in $G_i$ and at least $k'_j$ colours are used by vertices in $G_{i'}$. As $|C'_j| = 2k'_j-1$, there is a colour $c \in C'_j$ that are used by both vertices of $G_i$ and $G_{i'}$, but every vertex of $G_i$ is adjacent to every vertex of $G_{i'}$, a contradiction.	
	\end{proof}

\begin{claim}
	\label{cl2}
	For each index $i \in \{1,2,\ldots, q\}$,  $\sum_{j \in J_i}k'_j \ge k_i$. 
\end{claim}
\begin{proof}
	 For each $j \notin J_i$, there is a set $A_j$  of $k'_j$ colours from $C'_j$ not used by any vertex in $G_i$. 
	 Let $S_l= (S_{l,1}, S_{l,2}, \ldots, S_{l,p}) \in {\cal S}$ be an element with  $S_{l,j} = A_j$  for all $j \notin J_i$. Then vertices in the $l$th copy of $K_{k_i}$ in $G_i$ does not use any colour from 
	 $\bigcup_{j \notin J_i}C'_j$. In other words, the colours used by vertices from the $l$th copy of $K_{k_i}$ of $G_i$ are all from 
	 $\bigcup_{j \in J_i} S_{l,j}$. As vertices in the $l$th copy of $K_{k_i}$ of $G_i$ are coloured by distinct colours, we conclude that 
	  $$|\bigcup_{j \in J_i} S_{l,j}| = \sum_{j \in J_i} |S_{l,j}| = \sum_{j \in J_i} k'_j \ge k_i.$$ 
\end{proof}

	 Let $\lambda''=\{k''_1,k''_2, \ldots, k''_q\}$, where $k''_i =  \sum_{j \in J_i} k'_j$. Then $\lambda''$ is obtained from $\lambda$ by increasing some parts of $\lambda$, and $\lambda'$ is a refinement of $\lambda''$. Hence $\lambda \le \lambda'$. \qed

\begin{corollary}
	\label{cor1}
	Assume $\lambda$ and $\lambda'$ are partitions of $k$.  Then every $\lambda$-choosable graph is $\lambda'$-choosable if and only if $\lambda'$ is a refinement of $\lambda$. 
\end{corollary}

\section{$\lambda$-choosability of special graphs}
\label{sec-planar}

Colouring of planar graphs motivated many colouring concepts and problems. 
 In the language of $\lambda$-choosability, the four colour theorem says that every planar graph is $\{1,1,1,1\}$-choosable. Voigt showed that there are planar graphs that are not $\{4\}$-choosable. 
The other partitions of $4$ are $\{1,3\}, \{2,2\}$ and $\{1,1,2\}$. 
It is natural to ask whether every planar graph $G$ is $\lambda$-choosable for  any of these partitions $\lambda$ of $4$.  Choi and Kwon \cite{CK2017} defined a {\em $t$-common $k$-assignment} of a graph $G$ to be a $k$-list assignment $L$ of $G$ with $|\cap_{v \in V(G)}L(v)| \ge t$. In other words, a
$t$-common $k$-assignment is precisely a $\lambda$-assignment with $\lambda=\{1,1,\ldots, 1, k-t\}$, where the number $1$ has multiplicity $t$. Choi and Kwon \cite{CK2017} constructed a planar graph with a $1$-common $4$-assignment $L$ for which $G$ is not $L$-colourable. Very recently, Kemnitz and Voigt \cite{KV2018} constructed a planar graph $G$ and a $2$-common $4$-assignment $L$ of $G$ for which $G$ is not $L$-colourable. In other words, there are planar graphs that are not $\{1,1,2\}$-choosable. This implies that for any partition $\lambda$ of $4$ different from $\{1,1,1,1\}$, there is a planar graph which is not $\lambda$-choosable.
So the Four Colour Theorem is tight in the refined scale of $\lambda$-choosability. Nevertheless, many interesting problems concerning $\lambda$-choosability of subfamilies of planar graphs remains open.  
 
It was shown by Mirzakhani \cite{Mirzakhani} that there are $3$-chromatic planar graphs that are not $4$-choosable. In contrast to this result, we have the following observation.

\begin{observation}
	Every $3$-chromatic planar graph is $\{1,3\}$-choosable.
\end{observation}
\begin{proof}
	Assume $G$ is a $3$-chromatic planar graph and $L$ is a $\{1,3\}$-assignment of $G$.  Let $C_1 \cup C_2$ be a partition of $\bigcup_{v \in V(G)}L(v)$ such that $|L(v) \cap C_1|= 1$ and $|L(v) \cap C_2|= 3$ for every vertex $v$.  Let $V_1,V_2,V_3$ be a partition of $V(G)$ into three independent sets.   It is known 
	%\cite{Alon} 
	that bipartite planar graphs are $3$-choosable. Thus  there is a proper colouring $f$ of $G[V_1 \cup V_2]$ such that $f(v) \in L(v) \cap C_2$ for every $v \in V_1 \cup V_2$. For $v \in V_3$, let $f(v)$ be the unique colour in $L(v) \cap C_1$. Then $f$ is an $L$-colouring of $G$.  
\end{proof}

Another observation is about $\{2,2\}$-choosable planar graphs. 
 
 \begin{observation}
 	Assume $G$ is a plane graph and $G^*$ is the dual of $G$. Assume $G^*$ has a spanning Eulerian subgraph $H$ such that each face of $H$ is incident to at most two connected components of $H$, and moreover, if a face $F$ of $H$ is incident to two connected components of $H$, then there are an even number of edges of $G$ connecting these two components of $H$. Then $G$ is $\{2,2\}$-choosable. 
 \end{observation}
 \begin{proof}
 	Since $H$ is an Eulerian plane graph, its faces can be properly $2$-coloured. I.e., the faces of $H$ can be partitioned into two independent sets $A$ and $B$. If a face $F$ of $H$ is incident to one  connected component of $H$, then the subgraph of $G$ induced by vertices contained in $F$ is a tree. If a face $F$ of $H$ is incident to two connected components, then the subgraph of $G$ induced by vertices contained in $F$ is a connected uni-cyclic graph (i.e., contains exactly one cycle) and the length of the cycle is the number of edges in $G^*$ connecting the two connected components of $H$. If two faces $F_1, F_2$ are not adjacent in $H$, then no vertex of $G$ contained in $F_1$ is adjacent to a vertex of $G$ contained in $F_2$. Let $X$ be the set of vertices of $G$ contained in faces in $A$, and $Y$ be the set of vertices of $G$ contained in faces in $B$. Then $X \cup Y$ is a partition of vertices of $G$, and each component of $G[X]$ or $G[Y]$ is either a tree or a uni-cyclic graph, and moreover, all the cycles are of even lengths. Therefore each of $G[X]$ and $G[Y]$ is $2$-choosable. 
 	
 	Assume $L$ is a $\{2,2\}$-assignment of $G$, and $C_1 \cup C_2$ are the corresponding colour groups. Then there is an $L$-colouring of $G$ such that   vertices in $X$ are coloured by colours from $C_1$ and vertices in $Y$ are coloured by colours from $C_2$. 
 \end{proof}

Voigt \cite{Voigt} constructed the first non-4-choosable planar graph. A few other constructions of non-$4$-choosable graphs are given later, each with certain special feature \cite{CK2017,KV2018,Mirzakhani,Zhu}. 
Here we present a new construction of a non-$\{1,3\}$-choosable planar graph.  A graph $G$ is {\em uniquely $k$-colourable} if there is a unique partition of $V(G)$ into $k$ independent sets.
Assume $G$ is uniquely $k$-colourable, and $V_1, V_2, \ldots, V_k$ is the unique partition of $V(G)$ into $k$ independent sets. There are $k!$ ways of assigning the $k$ colours $\{1,2,\ldots, k\}$ to the independent sets. So there are actually $k!$ $k$-colourings of $G$. If $G$ is a uniquely $4$-colourable planar graph, then there are exactly $24$ $4$-colourings of $G$.

For a plane graph $G$, we denote by ${\cal F}(G)$ the set of faces of $G$.

\begin{lemma}
	\label{lemma-unique}
	There exists a uniquely $4$-colourable plane triangulation $G'$, a set ${\cal F}$ of $24$ faces of $G'$ and a one-to-one correspondence $\phi$ between ${\cal F}$ and the $24$ $4$-colourings of $G'$ such that for each $F \in {\cal F}$, $\phi_F(V(F)) = \{1,2,3\}$, where $\phi_F$ is the $4$-colouring of $G'$ corresponding to $F$ and $V(F)$ is the set of vertices incident to $F$.
\end{lemma}
\begin{proof}
	Build a plane triangulation which is uniquely $4$-colourable and which has $24$ faces. Such a graph can be constructed by starting  from a triangle $T=uvw$, and repeat the following: choose a face $F$
	(which is a triangle), add a vertex $x$ in the interior of $F$ and connect $x$ to each of the three vertices of $F$. Each iteration
	of this procedure increases the number of faces of $G$ by $2$. We stop when there are $24$ faces.
	
	Let $\phi$ be an arbitrary one-to-one correspondence between the $24$   $4$-colourings of $G$ and the $24$ faces of $G$.
	For each face $F$ of $G$, we denote by $\phi_F$ the corresponding $4$-colouring of $G$.
	
	Let ${\cal F}'$ be the set of faces $F$ for which $\phi_F(V(F)) \ne \{1,2,3\}$.
	
	For each $F \in {\cal F}'$, add a vertex $z_F$ in the interior of $F$, connect $z_F$ to each of the three vertices of $F$. The colouring $\phi_F$ is uniquely extended to $z_F$.   Hence the resulting plane triangulation $G'$ is still uniquely $4$-colourable.
	The face $F$ of $G$ is partitioned into three faces of $G'$. The vertices of one of the three faces  are coloured by $\{1,2,3\}$. We denote this face by $F'$ and  use this face of $G'$ instead of the face $F$ of $G$ to be associated with   the colouring $\phi_F$ (and we denote this colouring by $\phi_{F'}$ after this operation).
	
	Let $${\cal F}=\{F': F \in {\cal F}'\} \cup ({\cal F}(G)-{\cal F}').$$
	
	The   one-to-one correspondence   between ${\cal F}$ and the $24$ $4$-colourings of $G'$ defined above satisfies the requirements of the lemma.
\end{proof}

Now we are ready to construct a planar graph $G$ and a $\{1,3\}$-assignment $L$ of $G$ such that $G$ is not $L$-colourable.  

Let $G'$ be a    uniquely $4$-colourable plane triangulation,  and let   ${\cal F}$ be a set of $24$ faces of $G'$ and  $\phi$  a one-to-one correspondence   between ${\cal F}$ and the $24$ $4$-colourings of $G'$ so that for each $F \in {\cal F}$, $\phi_F(V(F))=\{1,2,3\}$.

For each face $F \in {\cal F}$,  for  $i \in \{1,2,3\}$, let $v_{F,i}$ be the vertex of $F$ with $\phi_F(v_{F,i})=i$.
Note that two faces $F, F' \in {\cal F}$ may share a vertex $v$. In this case, if $\phi_F(v)=i$ and $\phi_{F'}(v)=j$, then $v=v_{F,i}=v_{F',j}$.

\begin{itemize}
	\item   add a triangle $T_F = a_Fb_Fc_F$ in the interior of $F$;
	\item   connect $a_F$ to   $v_{F,1}$ and $v_{F,2}$; connect $b_F$ to   $v_{F,1}$ and $v_{F,3}$; connect $c_F$ to $v_{F,2}$ and $v_{F,3}$.
\end{itemize}

We denote the resulting plane triangulation by $G$. 

Let $L$ be the $4$-assignment of $G$ defined as follows: 

\[
L(v)=\begin{cases}
\{1,2,3,4\}, &\text{if $v \in V(G')$}, \cr
\{1,2,4,5\}, &\text{if $v=a_F$ for some $F \in {\cal F}$}, \cr
\{1,3,4,5\}, &\text{if $v=b_F$ for some $F \in {\cal F}$}, \cr
\{2,3,4,5\}, &\text{if $v=c_F$ for some $F \in {\cal F}$}.
\end{cases}
\]

The set of colours used in the lists is $C=\{1,2,3,4,5\}$ and 
$C_1 = \{4\}, C_2=\{1,2,3,5\}$ is a partition of $C$ and for every vertex $v$ of $G$, $|L(v) \cap C_1|=1$ and $|L(v) \cap C_2| =3$. So $L$ is a $\{1,3\}$-assignment of $G$. 

Now we show that $G$ is not $L$-colourable.

Assume $\psi$ is an $L$-colouring of $G$. Then the restriction of $\psi$ to $G'$ is a proper $4$-colouring of $G'$.
As $G'$ is uniquely $4$-colourable, the restriction of $\psi$ to $G'$
equals $\phi_F$ for some $F \in {\cal F}$. Consider the triangle $T_F$. Vertex $a_F$ is adjacent to   vertices of colours $1$ and $2$ and has list $L(a_F)=\{1,2,4,5\}$. Therefore $\phi(a_F) \in \{4,5\}$. 
Vertex $b_F$ is adjacent to   vertices of colours $2$ and $3$ and has list $L(b_F)=\{2,3,4,5\}$. Therefore $\phi(b_F) \in \{4,5\}$. 
Vertex $c_F$ is adjacent to   vertices of colours $1$ and $3$ and has list $L(c_F)=\{1,3,4,5\}$. Therefore $\phi(c_F) \in \{4,5\}$. This is a contradiction, as $a_Fb_Fc_F$ form a triangle, and hence cannot be coloured by colours $4,5$. 

This completes the proof that $G$ is not $L$-colourable. Hence $G$ is a planar graph which is not $\{1,3\}$-choosable. \qed

\bigskip

Given a assignment $L$ of a graph $G$, let $$||L|| =|\{L(x): x \in V(G)\}|.$$ 

The cardinality $||L||$ is another measure of the complexity of a assignment $L$.
The example above shows that there is a $4$-assignment  $L$ of a planar graph $G$ such that   $||L|| = 4$ and $G$ is not $L$-colourable. It is easy to see that if $L$ is a $4$-assignment of a planar graph with $||L|| \le 2$, then $G$ is $L$-colourable. Indeed, if $||L||=1$, then the statement is the same as the four colour theorem. If $||L||=2$, then without loss of generality, we may assume that $L(x) \in \{\{1,2,3,4\}, \{i, i+1,i+2,i+3\}\}$ for some $2 \le i \le 5$. Let $\phi: V(G) \to \{1,2,3,4\}$ be a $4$-colouring of $G$. Let $\psi(v)=c$, where $c$ is   the unique colour $c$ in $L(v)$
for which $c \equiv \phi(v) \pmod{4}$. It is easy to check that $\psi$ is an $L$-colouring of $G$. 

It remains an open question whether there is a planar graph  $G$ and a $4$-assignment $L$ of $G$ with $||L||=3$ such that $G$ is not $L$-colourable. 

\begin{question}
	\label{g6}
	Is it true that for any   planar graph $G$, and    $4$-assignment $L$ of $G$ with $||L||=3$,   $G$ is  $L$-colourable?
\end{question} 

The following Theorem connects this question to weakly $4$-choosable graphs and $\{1,1,2\}$-choosable graphs.

\begin{theorem}
	\label{thm-implies2}
	Assume $G$ is weakly $4$-choosable and also $\{1,1,2\}$-choosable. If $L$ is a $4$-assignment of $G$ with $||L||=3$, then $G$ is $L$-colourable.   
\end{theorem}
\begin{proof}
	Assume $G$ is a counterexample.
	I.e.,   $G$   is weakly $4$-choosable and   $\{1,1,2\}$-choosable, but there is  a $4$-assignment $L$ of $G$ with $||L||=3$ such that $G$ is not $L$-colourable.

	We choose the counterexample so that $|\bigcup_{v \in V(G)}L(v)|$ is minimum.
	
	Assume $A, B, C$ are $4$-sets and for each vertex $v$ of $G$, $L(v) \in \{A, B,C\}$. 
	
	First we observe that $(A - B) \cup (B-A) \subseteq C$, which implies that $|A \cup B| \le 6$.  Assume this is not true. By symmetry, we may assume that $A-B \not\subseteq C$. Hence $A \not\subseteq B \cup C$. 
	
	If $i \in A-(B \cup C)$, then let $i'$ be any colour in $(B \cup C)-A$ and let $A'=(A-\{i\} ) \cup \{i'\}$. Let $L'$ be the assignment of $G$ for which $L'(v) = A'$ if $L(v)=A$, and $L'(v)=L(v)$ otherwise. Then we have $||L'|| \le 3$.  Since 
	$|\bigcup_{v \in V(G)}L'(v)| <  |\bigcup_{v \in V(G)}L(v)|$, by our choice of $(G,L)$, we know that $G$ has an $L'$-colouring $f'$. Let $f(v)=i$ if $f'(v)=i'$ and $L(v)=A$, and $f(v)=f'(v)$ otherwise. It is easy to verify that $f$ is a proper $L$-colouring of $G$. A contradiction.

	If $|A \cup B |=5$, then $|A \cap B \cap C|=2$ and hence $L$ is a  $\{1,1,2\}$-assignment of $G$. By our assumption $G$ is $L$-colourable.
	
	If $|A \cup B |=6$, say $A=\{1,2,3,4\}$ and $B=\{1,2,5,6\}$, then  $ C =(A - B) \cup (B-A)  =\{3,4,5,6\}$. Thus colours in $L(v)$ for each vertex $v$ come in pairs; $(1,2), (3,4)$ and 
	$(5,6)$. Change the names of colours as follows: $2$ $\to$ $-1$, $4$ $\to$ $-3$ and $6$ $\to$ $-5$. Then $L$ is a symmetric $4$-assignment of $G$. 	 Hence $G$ is $L$-colourable.
\end{proof}

Next we consider $\lambda$-choosability of line graphs. Vizing's motivation for introducing the concept of list colouring of graphs was to study  list colouring of line graphs as a tool to establish the total chromatic number of graphs.  The following conjecture, known as   the List Colouring Conjecture (LCC),    was formulated independently by
Vizing, by Gupta, by Albertson and Collins, and by Bollob\'{a}s and Harris 
(see \cite{BKW}). We say a graph $G$ is {\em edge $k$-colourable} (respectively, {\em edge $k$-choosable}) if its line graph is $k$-colourable (respectively, $k$-choosable).

\begin{guess} [LCC]
	\label{conj-lcc}
	Every edge $k$-colourable   graph is edge $k$-choosable. 
\end{guess}

This conjecture received a lot of attention, however, progress is slow. In particular, it is unknown if the conjecture holds for complete graphs of even order.

The concept of $\lambda$-choosability suggests intermediate problems for many existing challenging open problems. 
If Conjecture \ref{conj-lcc} is true, then the $\lambda$-choosability problem of line graphs would collapse   to a colourability problem.

\begin{theorem}
	\label{thm-line}
	If $G$ is an edge $k$-colourable  graph and $\lambda=\{k_1, k_2, \ldots, k_q\}$ is a partition of $k$ in which each part has size at most $2$, i.e., $k_i \le 2$, then $G$ is edge  $\lambda$-choosable.
	If $n$ is an even integer and $\lambda=\{k_1, k_2, \ldots, k_q\}$ is a partition of $n-1$ in which each part has size at most $3$, i.e., $k_i \le 3$, then $K_n$ is edge $\lambda$-choosable.
\end{theorem}
\begin{proof}
		Assume  $G$ is an edge $k$-colourable  graph, $f$ is 
		a   $k$-colouring of the line graph of $G$,
		and $\lambda=\{k_1, k_2, \ldots, k_q\}$ is a partition of $k$ in which each part has size at most $2$. 
		Let $L$ be a $\lambda$-assignment of the line graph of $G$,
		and let $ \bigcup_{i=1}^q C_i$ be a partition of $\bigcup_{e \in E(G)}L(e)$ such that for each index $i$, for any edge $e$ of $G$,   $|L(e) \cap C_i| =k_i$.  Let $s_0=0$ and $s_i = s_{i-1}+k_i$. For $i=1,2,\ldots, q$, let $E_i = \bigcup_{j=s_{i-1}+1}^{s_i}f^{-1}(j)$. So either $k_i=1$ and hence $E_i$ is a matching and hence edge $1$-choosable, or $k_i=2$ and $E_i$ induces a subgraph of $G$ whose components are even cycles and paths and hence  is $2$-edge choosable.  
		Thus we can colour all the edges   $e \in E_i$ using colours from $L(e) \cap C_i$, and obtain an $L$-colouring of $E(G)$. 
		
	  Next we assume that $n$ is even,  $\lambda = \{k_1,k_2,\ldots, k_q\}$ is a partition of $n-1$ and each $k_i \le 3$, and $L$ is a $\lambda$-assignment of $L(K_n)$. 
	Similarly, assume  $\bigcup_{e \in E(G)}L(e) = \bigcup_{i=1}^q C_i$, where for each index $i$, for any edge $e$ of $K_n$,   $|L(e) \cap C_i| =k_i$. We shall construct an $L$-colouring of $E(K_n)$.  
	 
	 A classical method of constructing an  $(n-1)$-colouring of the line graphs of $K_n$ is as follows: Order the vertices of $K_{n-1}$ in a circle $v_1, v_2, \ldots, v_{n-1}$. Add one vertex $v_0$ in the center of the circle. Construct a matching $M_1$ as $v_0v_1, v_{n-1}v_2, v_{n-2}v_3, \ldots, v_{n-i}v_{i+1}, \ldots, v_{n/2+1}v_{n/2}$ (see the solid edges in Figure \ref{fig3} (a)).

\begin{figure}[h]
	\begin{center}
		\includegraphics[scale=0.6]{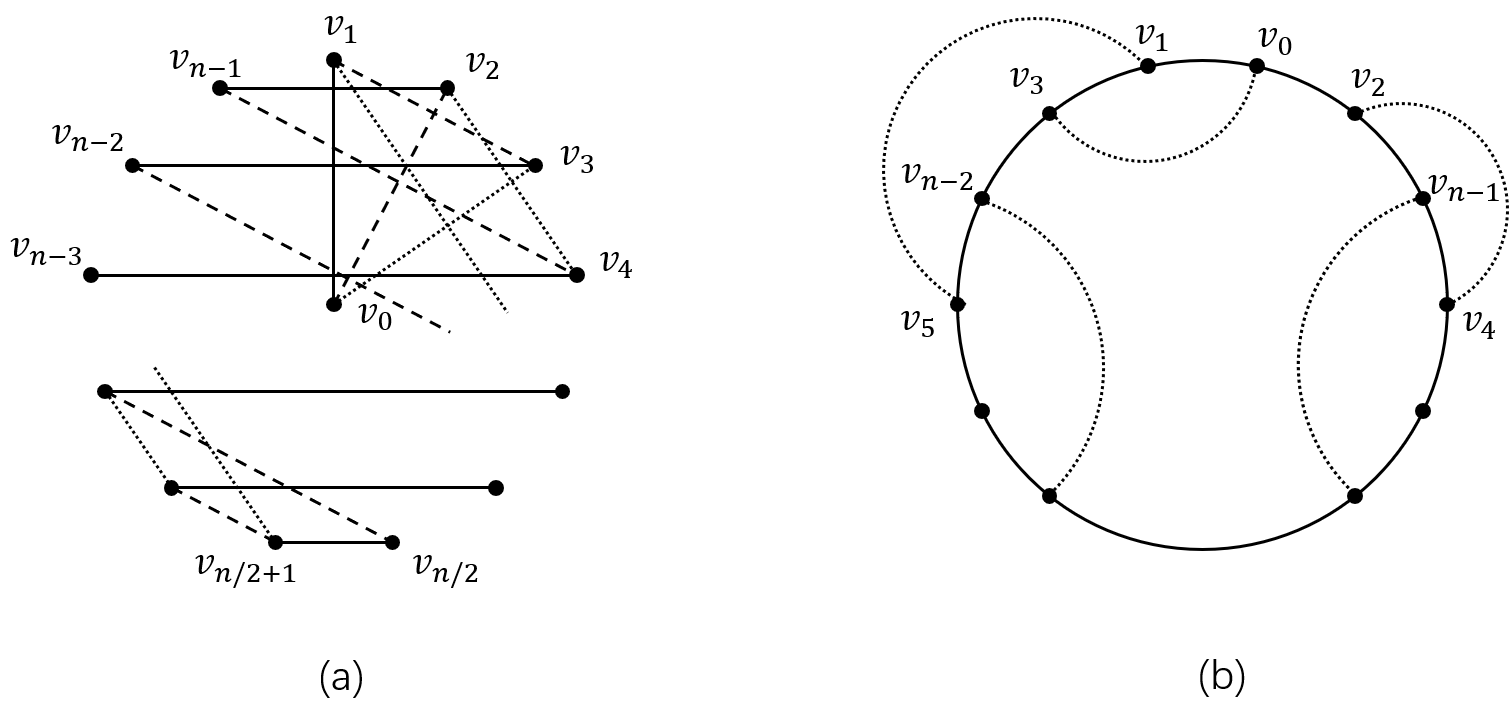} %\quad \quad
	\end{center}
	\caption{ (a) The matching $M_1 \cup M_2 \cup M_3$, (b) Embedding of $M_1 \cup M_2 \cup M_3$ in the plane.}
	\label{fig3}
\end{figure}

	  Rotate the matching clockwise for  $2\pi/(n-1)$ to get matching $M_2$, which consists of edges $v_0v_2, v_1v_3, v_{n-1}v_4, \ldots, v_{n-i}v_{i+3}, \ldots, v_{n/2+2}v_{n/2+1}$. 
	  Rotate another    $2\pi/(n-1)$ to get matching $M_3$, and matching $M_4$, $\ldots$, $M_{n-1}$. With each matching be a colour class, we obtain an $(n-1)$-colouring of the line graph of $K_n$. 
	 
	 The union $M_1 \cup M_2  $ is a Hamilton cycle and the union $M_1 \cup M_2 \cup M_3$ is a cubic planar graph  as shown in Figure \ref{fig3} (b).

	 It is well-known \cite{Dougbook} that $3$-edge colourable cubic plane graph is $3$-edge choosable. Thus the spanning subgraph of $K_n$ with edge set $M_1 \cup M_2 \cup M_3$ is $3$-edge choosable, as well as  the spanning subgraph of $K_n$ with edge set  $M_{i+1} \cup M_{i+2} \cup M_{i+3}$ for any index $i$,   is $3$-edge choosable.
	 
	   Let $s_0=0$ and $s_i = s_{i-1}+k_i$. By the argument above, the spanning subgraph of $K_n$ with edge set $E_i = \bigcup_{j=s_{i-1}+1}^{s_i}M_j$ is $k_i$-edge choosable. 
	 Thus we can colour all the edges   $e \in E_i$ using colours from $L(e) \cap C_i$, and obtain an $L$-colouring of $E(K_n)$. 
\end{proof}

It would be interesting to generalize Theorem \ref{thm-line} to all line graphs.

\section{Signed graph colouring and $\lambda$-choosability}
\label{sec-relation}

 A signed graph is a pair $(G, \sigma)$, where $G$ is a graph and $\sigma: E \to \{-1,+1\}$ is a   signature.
 For a positive integer $k$, let $Z_k$ be the cyclic group of order $k$, 
 and   $N_k=\{1, -1, 2, -2, \ldots, q, -q\}$ if 
 $k=2q$ is even and $N_k = \{0, 1, -1, 2, -2, \ldots, q, -q\}$ if $k=2q+1$ is odd. A    $k$-colouring  of $(G, \sigma)$ is 
a mapping $f: V(G) \to N_k$ such that for each edge $e=xy$, $f(x) \ne \sigma(e)f(y)$, and 
 a    $Z_k$-colouring  of $(G, \sigma)$ is a mapping $f: V(G) \to Z_k$ such that for each edge $e=xy$, $f(x) \ne \sigma(e)f(y)$. 
A graph $G$ is   signed $k$-colourable  (respectively,   signed  $Z_k$-colourable) if for any signature $\sigma$ of $G$, the signed graph $(G,\sigma)$ is $k$-colourable (respectively,  $Z_k$-colourable).

\begin{theorem}
	\label{thm-implies}
	Every signed $4$-colourable graph is weakly $4$-choosable.
\end{theorem}
\begin{proof}
	Assume $G$ is signed $4$-colourable and $L$ is a  symmetric $4$-assignment of $G$. Let $L^+(u) = \{|i|: i \in L(u)\}$ for each vertex $u$.

	We define a signature $\sigma$ of $G$ as follows: 
	For $e=uv \in E(G)$,   
	\[
	\sigma(e) =	\begin{cases}
	-1, & \text{ if $\min L^+(u) = \max L^+(v)$ or $\min L^+(v) = \max L^+(u)$}, \cr
	1, &\text{ otherwise.}
	\end{cases}
	\]
	
	By our assumption, $G$ is signed $4$-colourable. Let $f: V(G) \to \{\pm 1, \pm 2\}$ be a $4$-colouring of $(G, \sigma)$. We define an $L$-colouring $\phi$ of $G$ as follows:
	
	For $v \in V(G)$, let 
	\[
	\phi(v) =	\begin{cases}
	\max L^+(v), & \text{ if $f(v)=2$}, \cr
	- \max L^+(v), & \text{ if $f(v)=1$}, \cr
	\min L^+(v), &\text{ if $f(v) = -2$}, \cr
	- \min L^+(v), &\text{ if $f(v) = -1$}.
	\end{cases}
	\]
	
	Now we show that $\phi$ is a proper colouring of $G$. 
		Assume to the contrary that $e=uv$ is an edge of $G$, and $\phi(u) = \phi(v)$.  Let $i = |\phi(u)|$.
	
	Assume $e=uv$ is a positive edge. Then either $i = \min L^+(u)  = \min L^+(v)$ or $i = \max L^+(u)  = \max L^+(v)$. In any  case, 
	$f(u)  f(v) > 0$. Since $e$ is a positive edge, we have $f(u) \ne f(v)$. It follows from the definition that $\phi(u)  \phi(v) < 0$, hence $\phi(u) \ne \phi(v)$, a contradiction.
	
	Assume $e=uv$ is a negative edge. Then either $i = \min L^+(u)  = \max L^+(v)$ or $i = \max L^+(u)  = \min L^+(v)$. In any case, 
	$f(u)f(v) < 0$.     Since $e$ is a negative edge, we have $f(u) \ne -f(v)$. Hence  $|f(u)| \ne |f(v)|$, which implies that $\phi(u)  \phi(v) < 0$, hence $\phi(u) \ne \phi(v)$, a contradiction.
\end{proof}

The converse of Theorem \ref{thm-implies} is not true. The graph $K_{2,2,2,2}$ is $4$-choosable (and hence weakly $4$-choosable), but it is not signed $4$-colourable \cite{KKZ2018}.

It was conjectured by
M\'{a}\v{c}ajov\'{a},   Raspaud and \v{S}koviera \cite{MRS} that every planar graph is signed $4$-colourable, and conjectured by   K\"{u}ndgen and Ramamurthi \cite{KR2002} that every planar graph is weakly $4$-choosable. Theorem \ref{thm-implies} shows that M\'{a}\v{c}ajov\'{a},   Raspaud and \v{S}koviera's conjecture implies 
 K\"{u}ndgen and Ramamurthi's conjecture. However, very recently, Kardo\v{s} and Narboni \cite{KN} constructed a 
 planar graph which is not signed 4-colourable. The conjecture of K\"{u}ndgen and Ramamurthi remains open.

\begin{theorem}
	\label{thm-ks-112}
	Every signed $Z_4$-colourable graph is $\{1,1,2\}$-choosable.
\end{theorem}
\begin{proof}
	Assume $G$ is a signed $Z_4$-colourable graph and   $L$ is a $\{1,1,2\}$-assignment of   $G$. We may assume that colours in the lists are positive integers. By Lemma \ref{lem0}, we may assume that $\{ 1,2\} \subseteq \cap_{v \in V(G)} L(v)$.  For each vertex $v$, let $L'(v) = L(v) -\{1,2\}$.

		We define a signature $\sigma$ of $G$ as follows: 
		For $e=uv \in E(G)$,   
		\[
		\sigma(e) =	\begin{cases}
		-1, & \text{ if $\min L'(u) = \max L'(v)$ or $\min L'(v) = \max L'(u)$}, \cr
		1, &\text{ otherwise.}
		\end{cases}
		\]
		
		By our assumption, $G$ is $Z_4$-colourable. Let $f: V(G) \to Z_4$ be a $Z_4$-colouring of $(G, \sigma)$. We define an $L$-colouring $\phi$ of $G$ as follows:
		
		For $v \in V(G)$, let 
		\[
		\phi(v) =	\begin{cases}
		\max L'(v), & \text{ if $f(v)=3$}, \cr
		\min L'(v), &\text{ if $f(v) = 1$}, \cr
		1,   &\text{ if $f(v) = 0$}, \cr 
		2,   &\text{ if $f(v) = 2$}. 
		\end{cases}
		\]
		
		It is  obvious that $\phi$ is an $L$-colouring of $G$. 
		Now we show that $\phi$ is a proper colouring of $G$. 
		
		Assume to the contrary that $e=uv$ is an edge of $G$, and $\phi(u) = \phi(v)=i$.  It is obvious that $i \ne 1,2$, and hence $i \in L'(u) \cap L'(v)$. If $i = \max L'(u) = \max L'(v)$ or $i = \min L'(u) = \min L'(v)$, then $\sigma(e)=1$ and hence $f(u) \ne f(v)$. This is a contradiction, as $\phi(u) = \phi(v) = \max L'(u) = \max L'(v)$ or 
		$\phi(u) = \phi(v) = \min L'(u) = \min L'(v)$ implies that $f(u) = f(v)$.  Assume that $i = \max L'(u) = \min L'(v)$. Then $\sigma(e)=-1$. So $f(u) \ne -f(v)$ in $Z_4$. This is again in
		contrary to the definition of $\phi$.   
\end{proof}

Again the converse of Theorem \ref{thm-ks-112} is not true. It can be verified that $K_{2,2,2,2}$ is also not $Z_4$-colourable.

As mentioned earlier, Kemnitz and Voigt \cite{KV2018} showed that there are planar graphs that are not $\{1,1,2\}$-choosable. Hence by Theorem \ref{thm-ks-112}, there are planar graphs that are not signed $Z_4$-colourable. This refutes a conjecture of Kang and Steffen \cite{KS} which asserts that every planar graph  is signed $Z_4$-colourable.

In the following, we present a direct construction of a signed  planar graph $(G, \sigma)$ which is not $Z_4$-colourable. Indeed, this planar graph $G$ is a slight modification of a graph constructed by Wegner in 1973 \cite{Wegner}, which was used as an example of planar graph whose   vertex set   cannot be partitioned into $V_1 \cup V_2$ such that $G[V_1]$ is bipartite and $G[V_2]$ is a forest. Kemnitz and Voigt's example graph was motivated by Wegner's graph and is more complicated. 

Let $(H, \sigma)$ be the signed graph shown in Figure \ref{fig1}.

\begin{figure}[h]
	\begin{center}
		\includegraphics[scale=0.7]{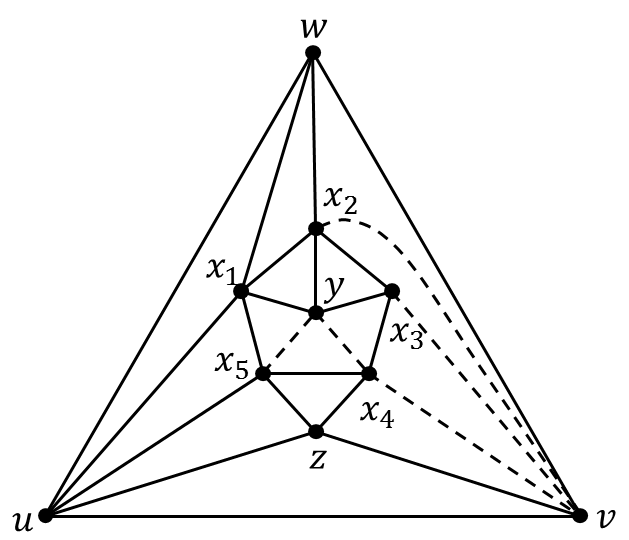} %\quad \quad
	\end{center}
	\caption{The signed graph $(H, \sigma)$, where the dotted lines are negative edges and solid lines are positive edges.}
	\label{fig1}
\end{figure}

\begin{claim}
	\label{clm1}
	For any $Z_4$-colouring $f$ of $(H, \sigma)$, $\{f(u), f(v)\} \cap \{0,2\} \ne \emptyset$. 
\end{claim}
\begin{proof}
	Assume to the contrary that $f$ is a $Z_4$-colouring of $(H, \sigma)$ and 
	$\{f(u), f(v)\} \cap \{0,2\} = \emptyset$.  By symmetry in colours, we may assume that $f(u)=1, f(v)=3$. 
	Then $f(w) \notin \{1,3\}$. By symmetry in colours again, we may assume that $f(w) = 0$. As the edge $vx_4$ is negative, $f(x_4) \ne -3=1$. Thus
	$f(z) \in \{0,2\}$ and $ f(x_4), f(x_5) \in \{0,2,3\}$. Since $z, x_4, x_5$ form a triangle with all edges positive, 
	$z, x_4, x_5$ are coloured by distinct colours. Hence either $f(x_4)=3$ or $f(x_5)=3$. 
	
	\bigskip
	\noindent
	{\bf Case 1} $f(x_5)=3$. 
	
	Now $x_1$ is adjacent to vertices of colours $0,1,3$ by positive edges, so $f(x_1)=2$.
	Then $x_2$ is adjacent to vertices of colours $0,2$ by positive edges, and to a vertex of colour $3$ by a negative edge. Hence $f(x_2) \ne 0,2,1$ and so $f(x_2)=3$. Similarly, these forces $f(y)=0$, which in turn forces $f(x_3)=2$. Then there is no legal colour for $x_4$, a contradiction.

	\bigskip
	\noindent
	{\bf Case 2} $f(x_4)=3$. 
	
	Then $f(x_3) \ne 3,1$. Hence $f(x_3)=0$ or $2$. If $f(x_3)=2$, then this forces $f(x_2)=3$, 
	which in turn forces $f(x_1)=2$ and $f(y)=0$. But then there is no legal colour for $x_5$, a contradiction. If $f(x_3)=0$, then we must have $\{f(x_1), f(x_2)\} = \{2,3\}$,  which leaves no legal colour for $y$, a contradiction.

	This completes the proof of the claim.
\end{proof}

We call the edge $uv$ of $H$ the {\em base edge} of $H$ and $w$ the {\em top vertex} of $H$. 
Let $(G, \sigma)$ be the signed planar graph as depicted in Figure \ref{fig2}, where in each of the six triangular faces, a copy of $(H, \sigma)$ is embedded. For each copy of $(H, \sigma)$, the three vertices on the boundary of $H$ are identified with the three vertices on the boundary of the corresponding face, and the top vertex $w$ is identified with the vertex  pointed by an arrow  inside the face. Note that the four vertices $u_1,u_2,u_3,u_4$ induce  a copy of $K_4$, and each of the six edges of this copy of $K_4$ is identified with the base edge of a copy of $(H, \sigma)$. 

\begin{figure}[h]
	\begin{center}
		\includegraphics[scale=0.7]{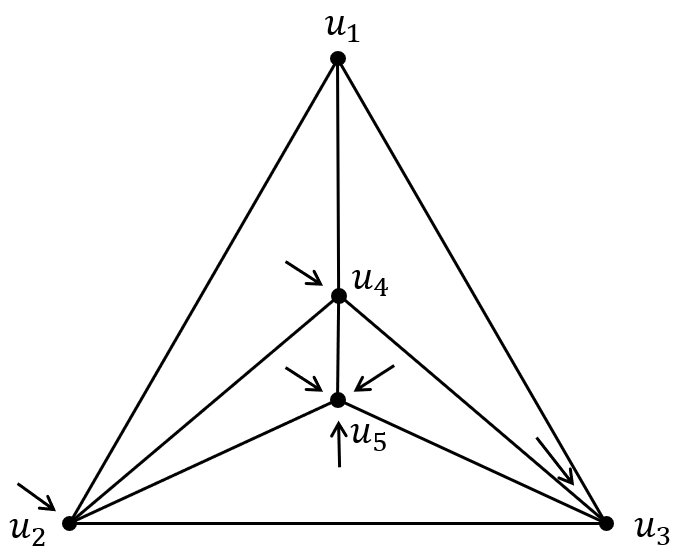} %\quad \quad
	\end{center}
	\caption{The signed graph $(G, \sigma)$, where all the edges shown in the figure are positive edges.}
	\label{fig2}
\end{figure}

\begin{theorem}
	\label{main}
	The signed planar graph $(G, \sigma)$ is not $Z_4$-colourable.
\end{theorem}
\begin{proof}
	Assume to the contrary that $f$ is a $Z_4$-colouring of $(G, \sigma)$. 
	
	By applying Claim \ref{clm1} to each of the six copies of $(H, \sigma)$, we conclude that  
	each of the six edges of the copy of $K_4$ induced by $u_1,u_2,u_3,u_4$ has an end vertex coloured by $0$ or $2$. On the other hand, all the edges of this copy of $K_4$ are positive, so the four vertices of this copy of $K_4$ are coloured by distinct colours from $Z_4$.	 This is a contradiction.
\end{proof}

Wegner's graph  is obtained from a copy of $K_4$ by adding six copies of $H$ and identifying the base edge  of each copy of $H$ with an edge of $K_4$,  and identifying three copies of the top vertex and a vertex of the $K_4$ (see \cite{Wegner}).  We could have used Wegner's graph instead of the graph $G$ described above. The proof of Theorem \ref{main} works for such a signed graph in the same way. The graph described above is obtained from Wegner's graph by identifying the top vertices of the other three copies of $H$ into a single vertex. This identification is not important, the only purpose is to make the graph a little smaller.

\section{Colouring of generalized signed graphs}
\label{sec-g}

  DP-colouring (also called correspondence colouring) of graphs is a
   concept introduced recently by Dvo\v{r}\'{a}k and Postle \cite{DP}, as a variation of list colouring of graphs, and has attracted a lot of attention (see \cite{BK,BKZ}).
   \begin{definition}
   	\label{def-dp} A $k$-cover of a graph $G$ is a graph $H$ with $V(H)=V(G) \times [k]$, in which each edge is of the form $(u,i)(v,j)$ for some edge $uv$ of $G$, and   for each edge $uv$ of $G$, 
   	edges of $H$ between $u\times [k]$ and $v \times [k]$ form a (not necessarily perfect) matching.
   	We say $G$ is DP-k-colourable if  any $k$-cover $H$ of $G$ has an independent set $I$ which intersects $v \times [k]$ exactly once for each vertex $v$ of $G$.  
   \end{definition}
    By using the concept of DP-colouring, 
  Dvo\v{r}\'{a}k and Postle \cite{DP} proved that planar graphs without cycles of lengths $4,5,6,7,8$ are $3$-choosable, solving a 15 year old open problem. DP-colouring can also be viewed as a generalization of colouring of signed graphs. This point of view is adopted in \cite{JWZ2018,JLYZ}, where the concept of generalized signed graphs is introduced. 
  
  In this section, we shall study colouring of generalized signed graphs that are strengthening of  $\lambda$-choosability of graphs. 
  
  First we define generalized signed graphs and their colouring.

For convenience, we view an undirected graph $G$ as  a symmetric digraph, in which each edge $uv$ of $G$ is replaced by two opposite arcs $e=(u,v)$ and $e^{-1}=(v, u)$. We denote by $E(G)$ the set of arcs of $G$.
A set  $S$   of permutations of 
positive integers is {\em inverse closed} if $\pi^{-1} \in S$ for every
$\pi \in S$.

\begin{definition}
	Assume $S$ is  an inverse closed subset  of permutations of 
	positive integers.
	An {\em $S$-signature} of $G$ is a mapping $\sigma: E(G) \to S$ such that for every arc $e$, $\sigma(e^{-1}) = \sigma(e)^{-1}$.
	The pair $(G, \sigma)$ is called an {\em $S$-signed graph}.
\end{definition}

\begin{definition}
	Assume $S$ is  an inverse closed subset  of permutations of 
	positive integers and $(G, \sigma)$ is an $S$-signed graph.
	A {\em $k$-colouring } of $(G, \sigma)$ is a mapping $f: V(G) \to [k]=\{1,2,\ldots, k\}$ such that for each 
	arc $e=(x,y)$ of $G$, 
	$\sigma(e)(f(x)) \ne f(y)$.
	We say $G$ is $S$-$k$-colourable if $(G, \sigma)$ is  $k$-colourable for every $S$-signature $\sigma$ of $G$.
\end{definition}

Colouring of generalized signed graphs is a common generalization of many colouring concepts. 
\begin{itemize}
	\item If  $S=\{id\}$, then $S$-$k$-colourable is equivalent to $k$-colourable.  
	\item  If $S=\{id, (12)(34)\ldots ((2q-1)(2q))\}$ when $q = \lfloor k/2 \rfloor$ or $q = \lceil k/2 \rceil -1$, then   $S$-$k$-colourable is equivalent to signed  $k$-colourable or signed  $Z_k$-colourable, respectively.
	\item If $S = <(12\ldots k)>$ is the cyclic group generated by permutation $(12\ldots k)$, then $S$-$k$-colourable is the same as $Z_k$-colourable, as defined by   
	Jaeger, Linial, Payan and Tarsi \cite{JLPT1992}. Indeed for each group $\Gamma$ of order $k$, there is a subgroup $S$ of the symmetric group $S_k$ such that $\Gamma$-colourable is equivalent to  $S$-$k$-colourable. 
	\item If $S $ is the set of all permutations, then   $S$-$k$-colourable is equivalent to DP-$k$-colourable.
\end{itemize}

It is shown in \cite{DP} that   every DP-$k$-colourable graph is $k$-choosable.   Indeed, assume  $L$ is a $k$-assignment of $G$. We define a signature $\sigma$ of $G$ as follows: For each edge $e=(x,y)$, let $\sigma(e)$ be any permutation of integers for which $\sigma(e)(i)=j$ if the $i$th colour in $L(x)$ equals the $j$th colour in $L(y)$. If the $i$th colour in $L(x)$ is not contained in $L(y)$, then $\sigma(i) \notin [k]$.  
Here we assume the colour set is ordered. Then $(G, \sigma)$ is $k$-colourable if and only if $G$ is $L$-colourable:
If $f$ is a $k$-colouring of $(G, \sigma)$, then let $\phi(x)$ be the $f(x)$th colour in $L(x)$. It is easy to verify that $\phi$ is an $L$-colouring of $G$. Conversely, if $\phi$ is an $L$-colouring of $G$, then let $f(x)=i$ if $\phi(x)$ is the $i$th colour in $L(x)$. It is easy to verify that $f$ is a $k$-colouring of $(G, \sigma)$.

In a very similar manner, the concept of $\lambda$-choosability is closely related  to  colouring of certain generalized signed graphs.

	Assume $\lambda = \{k_1,k_2,\ldots,k_q\}$ is a partition of $k$.
	Let $  I_1 \cup I_2 \cup \ldots \cup I_q$ be a partition of $[k]$, where $I_j= \{s_{j-1}+1,s_{j-1}+2,\ldots, s_{j-1}+k_j\}$ and $s_0=0$ and for $j \ge 1$, $s_j=s_{j-1}+k_j$.    We denote by
	$S_{\lambda}$ to be the set of   permutations $\sigma$ of $[k]$ %=\{1,2,\ldots, k\}$ 
	such that $\sigma(I_j)  =  I_j$ for $1 \le j \le q$.

\begin{theorem}
	\label{thm2}
		Assume $\lambda = \{k_1,k_2,\ldots,k_q\}$ is a partition of $k$. 
		If $G$ is a $S_{\lambda}$-$k$-colourable graph, then $G$ is $\lambda$-choosable. 
\end{theorem}
\begin{proof}
	Assume $G$ is a $S_{\lambda}$-$k$-colourable graph, and $L$ is a $\lambda$-assignment of $G$. 
	By definition, there is a partition of  $\bigcup_{v \in V(G)} L(v)$ as $ C_1 \cup C_2 \cup \ldots \cup C_q$  such that for each vertex $v$ and for each $1 \le i \le q$, $|L(v) \cap C_i| = k_i$. 
		We may assume the colours are ordered in such a way that if $i < j$ then every colour in $C_i$ is less than every colour in $C_j$, and every integer in $I_i$ is less than every integer in $I_j$. 
 
	Then for any $l \in I_j$ for any vertex $v$ of $G$, the $l$th colour of $L(v)$ belongs to $C_j$.

	We define a signature $\sigma$ of $G$ as follows: For each edge $e=(x,y)$, if $l \in I_j$ and the $l$th colour in $L(x)$  equals the $l'$th colour in $L(y)$,
	let $\sigma(e)(l)=l'$. If the $l$th colour in $L(x)$ is not equal to any colour in $L(y)$, then $\sigma(e)(l)$ is an arbitrary colour in $I_j$, provided that the resulting mapping $\sigma(e)$ is a permutation of  colours.
Thus for each $j$,  $\sigma(e)(I_j) =I_j$ and   $\sigma(e) \in S_{\lambda}$. Hence $(G, \sigma)$ has a $k$-colouring $f$. For each vertex $x$ of $G$, let $\phi(x)$ be the $f(x)$th colour in $L(x)$. Then it is easy to verify that $\phi$ is an $L$-colouring of $G$. 	
\end{proof}

 Note that if $\lambda = \{k\}$, then $S_{\lambda}$-$k$-colourable is the same as DP-$k$-colourable.
It is known that there are $k$-choosable graphs that are not DP-$k$-colourable. So the converse of Theorem \ref{thm2} is not true.

Similar to Lemma \ref{lem1}, we have the following lemma for $S_{\lambda}$-$k$-colourable graphs.

\begin{lemma}
	\label{lem2}
	Assume for $i=1,2,\ldots, q$, $\lambda_i$ is a partition of $k_i$, and
	$G_i$ is $S_{\lambda_i}$-$k_i$-colourable. 
	Let $\lambda  = \bigcup_{i=1}^q \lambda_i$ be the union of the multisets $\lambda_i$ (the multiplicity of an integer $s$ in $\lambda$ is the sum of its multiplicities in $\lambda_i$). Then $\vee_{i=1}^qG_i$ is $S_{\lambda}$-$k$-colourable.  
\end{lemma}
\begin{proof}
	Let $\sigma$ be a $S_{\lambda}$-signature of $G$. 
	For $j=1,2,\ldots, q$, let $k'_j = k_1+k_2+\ldots + k_{j-1}$ and 
	$I_j = \{k'_j+1,k'_j+2, \ldots, k'_j+k_j\}$. By definition, 
	$\sigma(I_j)=I_j$. Let $\sigma_j$ be the permutation of $[k_j]$
	defined as $\sigma_j(t) = \sigma(t+k'_j) - k'_j$. It follows from the definition that $\sigma_j 
	\in S_{\lambda_j}$. Hence 
	  $(G_j,\sigma_j)$ is $k_j$-colourable. Let $f_j$ be a $k_j$-colouring of $(G_j,\sigma_j)$. Let $f$ be the $k$-colouring of $G$ defined as
	  $
	  f(v) = f_j(v)+k'_j
	  $
	  for each vertex $v$ of $G_j$. 
	  Then it is easy to verify that $f$ is a $k$-colouring of $(G, \sigma)$.  
\end{proof}

\begin{corollary}
	\label{main2}
	Assume $\lambda$ is a partition of $k$ and $\lambda'$ is a partition of $k'$.
	If $\lambda \le \lambda'$, then every $S_{\lambda}$-$k$-colourable graph is $S_{\lambda'}$-$k'$-colourable, and conversely, if every $S_{\lambda}$-$k$-colourable graph is $S_{\lambda'}$-$k'$-choosable, then $\lambda \le \lambda'$. 	
\end{corollary}
\begin{proof}
	If 	$\lambda'$ is a refinement of $\lambda$, then any $S_{\lambda'}$-signature is an $S_{\lambda}$-signature. Hence  every $S_{\lambda}$-colourable graph is $S_{\lambda'}$-colourable. 
	
	Assume $k' > k$ and $\lambda'$ is obtained from $\lambda$ by increasing some parts of $\lambda$. We shall show that  every $S_{\lambda}$-$k$-colourable graph is $S_{\lambda'}$-$k'$-colourable. By using induction, it suffices to consider the case that $\lambda=\{k_1,k_2,\ldots, k_q\}$ and $\lambda'=\{k_1, k_2, \ldots, k_{q-1}, k_q+1\}$.  
	For any $S_{\lambda'}$-signature $\sigma'$ of $G$, if $\sigma'(k+1)=k+1$, then let $\sigma$ be the restriction of $\sigma'$ to $[k]$; if $\sigma'(i)=k+1, \sigma'(k+1)=j$, then let $\sigma$ be the restriction of $\sigma'$ to $[k]$, except that $\sigma(i)=j$.  Then a $k$-colouring of $(G, \sigma)$ is also a $k'$-colouring of $(G, \sigma')$.  Hence  every $S_{\lambda}$-$k$-colourable graph is $S_{\lambda'}$-$k'$-colourable.

	 For the converse direction, assume every $S_{\lambda}$-$k$-colourable graph is $S_{\lambda'}$-$k'$-colourable. Assume $\lambda =\{k_1,k_2,\ldots, k_q\}$.
	Let $G_i$ for $i=1,2,\ldots, q$ be the disjoint union of $n$ copies of complete graphs $K_{k_i}$ and let $G=\vee_{i=1}^q G_i$. Then each $G_i$ is $S_{k_i}$-$k_i$-colourable. By Lemma \ref{lem2}, $G$ is $S_{\lambda}$-$k$-colourable.   So $G$ is $S_{\lambda'}$-$k'$-colourable.

	By Theorem \ref{thm2}, $G$ is $\lambda'$-choosable. As shown in the proof of  Theorem \ref{main1}, when $n$ is sufficiently large, we must have $\lambda \le \lambda'$. 
\end{proof}

  \begin{definition}
  	\label{def-goodbad}
  	Assume $S$ is an inverse closed non-empty subset of $S_4$. We say $S$ is {\em good} if   each $i \in [4]$is fixed by some permutation in $S$ and every planar graph is $S$-$4$-colourable.
  \end{definition}
  
  We say two subsets $S$ and $S'$ of $S_k$ are conjugate if there is a permutation $\pi \in S_k$ such that $S'= \{\pi \sigma\pi^{-1}: \sigma \in S\}$. 
  The four colour theorem is equivalent to say that $S=\{id\}$ is good.  
  It is proved in \cite{JWZ2018} that $\{id\}$ is the only good subset of $S_4$ containing $id$, and it is proved in \cite{JZ2019} that, up to conjugation,  every good subset of $S_4$ not containing $id$ is a subset of $\{(12),(34), (12)(34)\}$.

\section{Some open problems}
\label{sec-open}

The concept of $\lambda$-choosability is a refinement of choosability, and basically all questions interesting for choosability are interesting with respect to $\lambda$-choosability. In particular, if a class of graphs are known to be $k$-colourable and not known to be or not  $k$-choosable, it is interesting to ask if they are $\lambda$-choosable for some  partitions $\lambda$ of $k$. There are many such questions. This section lists a few such questions. 

By the classical Gr\"otzsch Theorem, every triangle free planar graph is $3$-colourable. It was shown by Voigt \cite{Voigt95} that there are triangle free planar graphs that are not $3$-choosable. On the other hand, it is easy to see that every triangle free planar graph is $4$-choosable.  So  the problem  
of maximum chromatic number and maximum choice number of triangle free planar graphs  is completely solved. With the refined scale of choosability introduced above, a natural question  arises. 

\begin{question}
	 Is every triangle free planar graph   $\{1,2\}$-choosable?
\end{question}

It is known that every planar graph without cycles of lengths $4,5,6,7$ are $3$-colourable \cite{BGR2010, BGRS2005}. However, whether such planar graphs are $3$-choosable remains an open question. It was conjectured by Montassier in \cite{Montassier} that every planar graph without cycles of lengths $4,5,6$ are $3$-choosable.
On the other hand, it remains open whether every planar graph without cycles of lengths $4,5,6$ is $3$-colourable. It was conjectured by Steinberg from 1976 that   every planar graph without cycles of lengths $4,5$ is $3$-colourable, and the conjecture was refuted in \cite{CHKLS} in 2017.
One plausible revision of Steinberg's conjecture is that every planar graph without cycles of lengths $4,5,6$ is $3$-colourable.
The following conjecture is sandwiched between Montassier's conjecture and the possible revision of Steinberg's conjecture.

\begin{question}
	\label{g10}
	Is it  true that every   planar graph without cycles of lengths $4,5,6$ is  $\{1,2\}$-choosable?
\end{question}

On the other hand, it remains an open   question  whether every   planar graph without cycles of lengths $4,5,6,7$ is  $\{1,2\}$-choosable.

Although there exists non-$\{1,1,2\}$-choosable planar graph, it is likely that planar graphs forbidding some configurations are $\{1,1,2\}$-choosable.
The non-$\{1,1,2\}$-choosable planar graphs given in \cite{KV2018}  contains $K_4$. By modifying that example, one can construct a $K_4$-free planar graph which is not $\{1,1,2\}$-choosable. Let $W_k$ be the  $k$-wheel, which is obtained from the cycle $C_k$ by adding a universal vertex. There are non-$4$-choosable planar graphs that contains no odd wheels. However the following question is open.

\begin{question}
	Is it true that every  planar graph containing no odd wheel is $\{1,1,2\}$-choosable?
	Or even $\{1,3\}$-choosable or $\{2,2\}$-choosable? 
\end{question}

We have shown that $3$-chromatic planar graphs are $\{1,3\}$-choosable. However, the following question remains open.

 \begin{question}
 	Is it true that every $3$-chromatic planar graph is $\{2,2\}$-choosable? 
 \end{question}

The List Colouring Conjecture is a very difficult conjecture. With the concept of $\lambda$-choosability, one may consider some weaker  version of this conjecture. The following conjecture is such an example.

 \begin{guess}
 	\label{conj-wlcc}
 	For any integer $s$, there is an integer $k(s)$ such that if $k \ge k(s)$, $G$ is edge  $k$-colourable, and $\lambda$ is a partition of $k$ in which each part has size at most $s$, then $G$ is edge $\lambda$-choosable. 
 \end{guess}

For the order of partitions of integers, the following question remains open.

\begin{question}
Assume $\lambda_1, \lambda_2, \lambda_3$ are   incomparable partitions of integers. Is it true that there is a graph which is $\lambda_i$-choosable for $i=1,2$ but not $\lambda_3$-choosable? 	 More generally, assume $A$ and $B$ are families of partitions of integers, and  the partitions in $A \cup B$ are pairwise incomparable. Is it true that there is a graph which is $\lambda$-choosable for all $\lambda \in A$ but not $\lambda$-choosable for all $\lambda \in B$? 	
\end{question}

%% main text


\begin{thebibliography}{99}\setlength{\itemsep}{-0.001mm}
		
		
		
		\bibitem{BK} A.~Bernshteyn and A.~Kostochka. \emph{On differences between DP-coloring and list coloring},  \text{https://arxiv.org/abs/1705.04883}, (Russian) Mat. Tr. 21 (2018), no. 2, 61–71.
		
		
		\bibitem{BKZ} A.~Bernshteyn and A.~Kostochka and X.~Zhu. \emph{DP-colorings of graphs with high chromatic number}. European J. Combin. 65 (2017), 122--129. 
		
	
		
		\bibitem{BGR2010}
		O.V. Borodin, A.N. Glebov, A. Raspaud, {\em 
		Planar graphs without triangles adjacent to cycles of length from 4 to 7 are 3-colorable},
		Discrete Math., 310 (2010),   2584-2594
		
		
		\bibitem{BGRS2005}
		O.V. Borodin, A.N. Glebov, A. Raspaud, M.R. Salavatipour, {\em 
		Planar graphs without cycles of length from 4 to 7 are 3-colorable},
		J. Combin. Theory Ser. B, 93 (2005),   303-331.
		
			\bibitem{BKW}
			O.V. Borodin, A.V. Kostochka and D.R. Woodall,{\em  List edge and list total colourings of multigraphs}, J. Combin. Theory Ser. B 71 (1997) 184–204.
		
		 \bibitem{CK2017} 
		 H. Choi and Y. Kwon, {\em 
		 On t-common list-colorings},  
		 Electron. J. Combin. 24 (2017), no. 3, Paper 3.32, 10 pp. 
		
	 
	 	\bibitem{CHKLS}
	 	V. Cohen-Addad, M. Hebdige, D. Kr\'{a}l, Z. Li and E. Salgado,
	 	{\em Steinberg's Conjecture is false }, J. Combin. Theory Ser. B, 122(2017), 452-456.  
	 	
	 
	 \bibitem{DP} Z.~Dvo\v{r}\'{a}k and L.~Postle. \emph{List-coloring embedded graphs without cycles of lengths $4$ to $8$}, J. Combin. Theory Ser. B 129 (2018), 38–54. 
	 
	 
	 
	 
	 \bibitem{ERT} P.~Erd\H{o}s, A.L.~Rubin, and H.~Taylor. \emph{Choosability in graphs}, Proc. West Coast Conf. on Combinatorics, Graph Theory and Computing, Congressus Numerantium XXVI, 1979, 125--157.
	 
	 	\bibitem{JLPT1992}
	 	F. Jaeger, N. Linial, C. Payan, and M. Tarsi, \emph{  Group connectivity of graphs—
	 		A non-homongenous analogue of nowhere-zero flow},  J. Combin. Theory Ser. B 
	 	56 (1992), 165 -- 182.
	 	
	 \bibitem{JLYZ} Y. Jiang, D. Liu, Y. Yeh and X. Zhu, {\em Colouring of generalized signed triangle free planar graphs}, Discrete Math. 342 (2019), no. 3, 836–843.
	 
	 \bibitem{JZ2019} Y. Jiang  and X. Zhu, {\em $4$-colouring of generalized signed   planar graphs}, manuscript, 2019.
	 
		\bibitem{JWZ2018} L. Jin, T. Wong and X. Zhu, \emph{ Colouring of generalized signed graphs}, submitted, 2018.  
		
%	\bibitem{JZ2019b} Y. Jiang  and X. Zhu, {\em Multiple list colouring triangle free planar graphs},  J. Combin. Theory Ser. B,  https://doi.org/10.1016/j.jctb.2018.12.004
	
		
		
		\bibitem{KS2} Y. Kang and E. Steffen, \emph{ The chromatic spectrum of signed graphs}, Discrete Mathematics 339 (2016) 2660--2663.
		
		\bibitem{KS} Y. Kang and E. Steffen, \emph{ Circular coloring  of signed graphs},   J. Graph Theory. 2017;00, 1-4. DOI: 10.1002/jgt.22147.
		
		\bibitem{St}
		Y. Kang and E. Steffen, {\em personal communication}, 2017. 
		
		\bibitem{KN} 	F. Kardo\v{s} and J. Narboni, \emph{
			On the 4-color theorem for signed graphs}, 	arXiv:1906.05638
		
		 \bibitem{KV2018} A. Kemnitz and  M. Voigt, \emph{
		 	A note on non-4-list colorable planar graphs}, Electronic  Journal of Combinatorics,  25(2) (2018), \#P2.46 
		 
		 
		\bibitem{KKZ2018} R. Kim, S. Kim and X. Zhu, {\em Signed colouring and list colouring of $k$-chromatic graphs}, manuscript, 2018.
		
		 
		
	 
		
		
		\bibitem{KR2002} A. K\"{u}ndgen, R. Ramamurthi, {\em Coloring face-hypergraphs of graphs on surfaces}, J. Combin. Theory Ser. B 85 (2002), 307–337.
		
	
		  
		 
		 
		 
		 \bibitem{MRS} E. M\'{a}\v{c}ajov\'{a}, A. Raspaud, M. \v{S}koviera,  \emph{ The chromatic number of a signed graph},
		 Electron. J. Combin. 23 (1) (2016) \#P1.14.
		 
		 
			\bibitem{Mirzakhani} 	M. Mirzakhani, {\em A small non-4-choosable planar graph}, Bull. Inst. Combin. Appl.
			17 (1996), 15–18.
			
	
		
		\bibitem{Montassier} M. Montassier, {\em  A note on the not 3-choosability of some families of planar graphs},
		Information Processing Letters 99 (2006) 68--71.
		
		 	
		 	
		 	
		 	
		
		\bibitem{Tho1994}
		C. Thomassen, {\em Every planar graph is 5-choosable},  J. Combin. Theory Ser. B   62(1):180--181, 1994.
	
%	\bibitem{Sternberg}
%	R. Steinberg, {\em The state of the three color problem, quo vadis, graph theory?}, Ann. Discrete Math. 55 (1993) 211–248.
	   	
	  
	   	
	
	\bibitem{Vizing}
	V.G. Vizing, \emph{Coloring the vertices of a graph in prescribed colors}, Diskret. Analiz No. 29 Metody
	Diskret. Anal. v Teorii Kodov i Shem (1976) 3-10, 101 (in Russian).
		
	  
	   
	  
	  
	 \bibitem{Voigt} M. Voigt, \emph{  List colourings af planar graphs}, Discrete Math., 120 (1993), 215--910.
	 
	 \bibitem{Voigt95} M. Voigt, {\em  A not $3$-choosable planar graph without $3$-cycles},  Discrete Math. 146 (1995), no. 1-3, 325–328.
	
		\bibitem{Wegner}  G. Wegner, \emph{ Note on a paper of B. Gr\"{u}nbaum on acyclic colorings}, Israel J. Math., 14: 409–412, 1973.
		
		\bibitem{Dougbook}
	D.	West, 
	{\em 	Introduction to graph theory},  Prentice Hall, Inc., Upper Saddle River, NJ, 1996. 
		
	  \bibitem{Z} T. Zaslavsky, \emph{ Signed graph coloring}, Discrete Math. 39 (1982) 215--228
	  
 
 
\bibitem{Zhu} X. Zhu, \emph{Multiple list colouring of planar graphs}, J. Combin. Theory Ser. B 122 (2017), 794–799.

 
	\end{thebibliography}
\end{document}